\documentclass[11pt,letterpaper,reqno]{amsart}
\newcommand{\ppar}[1]{\par\addvspace{\medskipamount}\noindent{\bfseries #1.}\enspace\ignorespaces}
\usepackage{amsmath,amssymb,amsthm,amsfonts,mathtools}
\usepackage{mathrsfs}
\usepackage{xcolor}
\usepackage[T1]{fontenc}
\usepackage{microtype}
\usepackage{aliascnt}

\usepackage[colorlinks=true,
  linkcolor=blue,
  citecolor=blue,
  urlcolor=blue]{hyperref}
\usepackage[nameinlink,capitalize,noabbrev]{cleveref}

\allowdisplaybreaks
\numberwithin{equation}{section}

\newtheorem{theorem}{Theorem}[section]

\newaliascnt{problem}{theorem}

\aliascntresetthe{problem}

\newaliascnt{proposition}{theorem}

\aliascntresetthe{proposition}

\newaliascnt{lemma}{theorem}
\newtheorem{lemma}[lemma]{Lemma}
\aliascntresetthe{lemma}

\newaliascnt{corollary}{theorem}

\aliascntresetthe{corollary}

\newaliascnt{remark}{theorem}

\aliascntresetthe{remark}

\crefname{theorem}{Theorem}{Theorems}
\Crefname{theorem}{Theorem}{Theorems}
\crefname{problem}{Problem}{Problems}
\Crefname{problem}{Problem}{Problems}
\crefname{proposition}{Proposition}{Propositions}
\Crefname{proposition}{Proposition}{Propositions}
\crefname{lemma}{Lemma}{Lemmas}
\Crefname{lemma}{Lemma}{Lemmas}
\crefname{corollary}{Corollary}{Corollaries}
\Crefname{corollary}{Corollary}{Corollaries}
\crefname{remark}{Remark}{Remarks}
\Crefname{remark}{Remark}{Remarks}

\newaliascnt{claim}{theorem}

\aliascntresetthe{claim}

\newaliascnt{question}{theorem}

\aliascntresetthe{question}

\newaliascnt{conjecture}{theorem}
\newtheorem{conjecture}[conjecture]{Conjecture}
\aliascntresetthe{conjecture}

\theoremstyle{definition}

\newaliascnt{definition}{theorem}
\newtheorem{definition}[definition]{Definition}
\aliascntresetthe{definition}

\newaliascnt{example}{theorem}

\aliascntresetthe{example}

\crefname{claim}{Claim}{Claims}
\Crefname{claim}{Claim}{Claims}
\crefname{question}{Question}{Questions}
\Crefname{question}{Question}{Questions}
\crefname{conjecture}{Conjecture}{Conjectures}
\Crefname{conjecture}{Conjecture}{Conjectures}
\crefname{definition}{Definition}{Definitions}
\Crefname{definition}{Definition}{Definitions}
\crefname{example}{Example}{Examples}
\Crefname{example}{Example}{Examples}

\newcommand{\one}{\mathbf{1}}

\newcommand{\normF}[1]{\left\lVert #1\right\rVert_F}

\newcommand{\cK}{\left(1-\frac1\omega\right)}

\begin{document}

\title{A positive square-energy strengthening of Tur\'an's theorem}

\author[Y.~Liu]{Yinchen Liu}
\address{Institute for Interdisciplinary Information Sciences, Tsinghua University, Beijing 100084, P. R. China}
\email{liuyinch23@mails.tsinghua.edu.cn}

\author[Q.~Tang]{Quanyu Tang}
\address{School of Mathematical Sciences, University of Science and Technology of China, Hefei 230026, P. R. China}
\email{tangquanyu827@gmail.com}

\author[S.~Zhang]{Shengtong Zhang}
\address{Department of Mathematics, Stanford University, Stanford, CA 94305, USA}
\email{stzh1555@stanford.edu}



\begin{abstract}
Let $G$ be an $n$-vertex graph with clique number $\omega(G)$, and let $s^+(G)$ denote the sum of the squared positive adjacency eigenvalues. We prove that
$$
\sqrt{s^+(G)}\le\left(1-\frac{1}{\omega(G)}\right)n.
$$
This strengthens Wilf's classical spectral Tur\'{a}n theorem and resolves a conjecture of Elphick and Wocjan. Adopting the relaxation of our companion paper on the square-energy conjecture, we reduce the theorem to a Motzkin--Straus inequality for doubly nonnegative matrices, which we prove via a local inverse-probability estimate for the Caro--Wei greedy algorithm on the complement. 
\end{abstract}

\maketitle

\section{Introduction}

Throughout the paper, graphs are finite and simple, and a graph is
\emph{nonempty} if it contains at least one edge.  Let $G=(V,E)$ be a
graph on $n$ vertices and $m$ edges, let $A$ be its adjacency matrix, and
write $\lambda_1\ge \lambda_2\ge\cdots\ge\lambda_n$ for the eigenvalues of
$A$.  For $v\in V$, let $N_G(v)$ denote the open neighborhood of $v$, and
write $d_G(v)=|N_G(v)|$.  The clique number of $G$ is denoted by
$\omega=\omega(G)$.

Mantel's theorem, proved in 1907, determines the maximum number of edges in a
triangle-free graph \cite{Mantel1907}. Tur\'an's 1941 theorem extended this
result to $K_{r+1}$-free graphs for arbitrary $r$ \cite{Turan1941}: let
$T_r(n)$ denote the complete $r$-partite graph on $n$ vertices whose part
sizes differ by at most one, and let $t_r(n)=|E(T_r(n))|$. Then every
$K_{r+1}$-free graph on $n$ vertices has at most $t_r(n)$ edges. In
particular, a graph with clique number $\omega$ satisfies
\[
m\le t_\omega(n)\le \frac12\left(1-\frac1\omega\right)n^2.
\]

A continuous formulation of Tur\'an's theorem was obtained by Motzkin and
Straus \cite{MotzkinStraus1965}:
\[
  \max\bigl\{x^{\top}Ax:x\ge0,\ \one^{\top}x=1\bigr\}
  =1-\frac1\omega,
\]
where $\one$ denotes the all-ones vector. Applying the Motzkin--Straus
inequality to a nonnegative principal eigenvector, Wilf~\cite{Wilf1986}
deduced the spectral clique bound
\begin{equation}\label{eq:wilf}
  \lambda_1\le \left(1-\frac1\omega\right)n.
\end{equation} 
Since $\lambda_1\ge 2m/n$, inequality \eqref{eq:wilf} recovers Tur\'an's
theorem in the form displayed above.

Nikiforov~\cite[Theorem~2.1]{Nikiforov2002} proved the edge-sensitive refinement
\begin{equation}\label{eq:nikiforov}
  \lambda_1^2\le 2m\left(1-\frac1\omega\right),
\end{equation}
confirming a conjecture of Edwards and
Elphick~\cite{EdwardsElphick1983}. Bollob\'as and
Nikiforov later conjectured that, except for complete graphs, the left-hand
side of \eqref{eq:nikiforov} can be replaced by
$\lambda_1^2+\lambda_2^2$ \cite[Conjecture~1]{BollobasNikiforov2007}.

The \emph{positive and negative square energies} of $G$ are
\[
  s^+(G):=\sum_{\lambda_i>0}\lambda_i^2,
  \qquad
  s^-(G):=\sum_{\lambda_i<0}\lambda_i^2.
\]
Wocjan and Elphick~\cite{WocjanElphick2013} introduced these quantities
in a conjectural strengthening of Hoffman's spectral lower bound on the
chromatic number~\cite{Hoffman1970}; their conjecture was proved by Ando
and Lin~\cite{AndoLin2015}. Elphick, Farber, Goldberg, and
Wocjan~\cite{ElphickFarberGoldbergWocjan2016} later initiated the systematic study of square energies and conjectured that
every connected graph $G$ on $n$ vertices satisfies
$\min\{s^+(G),s^-(G)\}\ge n-1$, recently resolved in~\cite{LiuTangZhang2026}. Positive and negative square
energies have since been studied as graph parameters in their own
right, through conic optimization, and as part of the broader theory
of $p$-energies; see
\cite{AbiadEtAl2023,AkbariEtAl2025Linear,AkbariEtAl2025PEnergy,AkbariEtAl2025Refinement,CoutinhoSpierZhang2024,ElphickLinz2024,ElphickTangZhang2026,TangLiuWang2026,Zhang2024}.

Since $s^+(G)\ge \lambda_1^2$, it is natural to ask whether Wilf's
inequality~\eqref{eq:wilf} remains valid with $\lambda_1$ replaced by
$\sqrt{s^+(G)}$.\footnote{The same bound for $s^-$ is trivial: since
$s^-=2m-s^+\le 2m-\lambda_1^2$ and $\lambda_1\ge 2m/n$, we get
$s^-(G)\le\frac{n^2}4$.}  The following conjecture was first stated by Elphick
and Wocjan in 2018~\cite{ElphickWocjan2018} and later appeared
as Conjecture~1 in the published paper of Elphick, Linz, and
Wocjan~\cite{ElphickLinzWocjan2024}.

\begin{conjecture}\label{conj:EW}
For every graph $G$ on $n\ge 1$ vertices,
\[
  \frac{n}{n-\sqrt{s^+(G)}}\le \omega(G).
\]
\end{conjecture}

Elphick, Linz, and Wocjan proved \cref{conj:EW} for triangle-free graphs,
weakly perfect graphs, Kneser graphs, and almost all graphs
\cite{ElphickLinzWocjan2024}.  Further classes were subsequently treated
by Jadav, Madyastha, Raut, and Singh~\cite{JadavEtAl2025}.  The conjecture
was also recorded as Conjecture~2 in the survey of Liu and
Ning~\cite{LiuNing2023}.

Our main result resolves \cref{conj:EW} in the affirmative.  We state
it in the following equivalent form.

\begin{theorem}\label{thm:main}
For every graph $G$ on $n\ge1$ vertices with clique number $\omega$,
\[
  \sqrt{s^{+}(G)}
  \le \left(1-\frac1\omega\right)n.
\]
\end{theorem}

The bound is tight: when $\omega$ divides $n$, the Tur\'an graph
$T_\omega(n)$ has exactly one positive eigenvalue, equal to
$\bigl(1-\frac1\omega\bigr)n$, so equality holds in \cref{thm:main}.

As a byproduct, our argument establishes the following weighted
Tur\'an stability
statement, in the sense of Erd\H{o}s and
Simonovits~\cite{Simonovits1968}, which may be of independent interest.  For a
partition $\mathcal P$ of $V(G)$ and nonnegative edge weights
$a=(a_e)_{e\in E(G)}$, define the \emph{separated mass}
\[
  W_{\mathcal P}(a)
  :=\sum_{\substack{uv\in E(G)\\
      u,\,v\text{ in different parts of }\mathcal P}} a_{uv}.
\]

\begin{theorem}
\label{thm:partition}
For every nonempty graph $G$ on $n$ vertices with clique number
$\omega$,
there is a probability distribution $\mathcal{D}_{\mathcal{P}}$ over
partitions of $V(G)$ into at most $\omega$ parts, such that for
\emph{any} choice of nonnegative edge weights $a_e$, we have
\[
  \mathbb{E}_{\mathcal P\sim\mathcal{D}_{\mathcal{P}}}\,W_{\mathcal P}(a)
  \;\ge\;
  \frac{2}{\bigl(1-\frac1\omega\bigr)n^2}
  \Bigl(\sum_{e\in E(G)}\sqrt{a_e}\Bigr)^{\!2}.
\]
\end{theorem}

Taking $a_e \equiv1$ shows that every nonempty graph $G$ with clique
number $\omega$ has an $\omega$-partite subgraph with at least
$\frac{2m^2}{(1-\frac1\omega)n^2}$ edges.  This strengthens, up to an
$O_{\omega}(1)$ additive rounding error, the seminal result of
F\"uredi~\cite{Furedi2015}, which produces an $\omega$-partite
subgraph with at least $m-t$ edges when $m=t_\omega(n)-t$.  Much stronger bounds are known very close to the
extremal density: for $t=\alpha n^2$ with $\alpha$ sufficiently small
in terms of $\omega$, Balogh, Clemen, Lavrov, Lidick\'y, and
Pfender~\cite{BaloghEtAl2021} used stability arguments to show that
$G$ even has an $\omega$-partite subgraph with at least
$m-O_\omega(\alpha^{3/2})\,n^2$ edges, the exponent $3/2$ being
sharp; see also~\cite{HuEtAl2021} for related bounds obtained by
flag algebra computations.  Closer to our method, Kelly and
Postle~\cite{KellyPostle2024} strengthened the Caro--Wei bound and
deduced that every $K_{\omega+1}$-free graph with
$m\geq \frac12\bigl(1-\frac1\omega\bigr)n^2-\frac{n}{2\omega}+1$ is
already $\omega$-partite. All of these results concern the
near-extremal regime, whereas \cref{thm:partition} applies to every
graph and the random partition is generated by an explicit (and in
fact, polynomial-time) random greedy process.

\subsection{Overview of the proof}
The proof was found by successively reducing \cref{thm:main} to stronger
and increasingly combinatorial statements, along the chain
\begin{gather*}
  \text{Spectral Tur\'an}
  \;\Rightarrow\;
  \text{DNN Motzkin--Straus}
  \;\Rightarrow\;
  \text{Caro--Wei partition and weighted Tur\'an}\\
  \;\Rightarrow\;
  \text{Local harmonic estimate}.
\end{gather*}
We now walk through this chain and state the key ideas of the paper
where they arise.

\ppar{Relaxation via a DNN Motzkin--Straus inequality}
Recall that a symmetric matrix $M$ is \emph{doubly nonnegative} (DNN) if $M\succeq 0$ (positive semidefinite) and $M\geq 0$ (entrywise nonnegative).  The first
step follows the same relaxation used in \cite{LiuTangZhang2026}, the resolution of the square-energy conjecture of Elphick, Farber, Goldberg, and Wocjan~\cite[Conjecture~1]{ElphickFarberGoldbergWocjan2016}.  Write
$A=A^+-A^-$, where $A^+,A^-\succeq0$ are the positive and negative
parts of the spectral decomposition of $A$, so that $A^+A^-=0$ and
$s^{+}(G)=\normF{A^{+}}^2$.
By the Schur product theorem \cite[Theorem~7.5.3]{HornJohnson2013}, the
entrywise square $M=A^{+}\circ A^{+}$ is doubly nonnegative, and it satisfies
\[
  \one^{\top}M\one=\normF{A^{+}}^2=s^{+}(G),
  \qquad
  2\sum_{uv\in E(G)}\sqrt{M_{uv}}
  =2\sum_{uv\in E(G)}\lvert A^{+}_{uv}\rvert
  \ \ge\ \operatorname{tr}(AA^{+})=s^{+}(G).
\]
Applying the following DNN analogue of the Motzkin--Straus inequality
to this particular $M$ therefore yields
$s^{+}(G)\le\bigl(1-\frac1\omega\bigr)n\sqrt{s^{+}(G)}$,
which is \cref{thm:main}.

\begin{lemma}\label{lem:psd}
For every graph $G$ on $n\ge1$ vertices with clique number $\omega$ and
every doubly nonnegative matrix $M$ indexed by $V(G)$, we have
\begin{equation}\label{eq:psd-ms}
  2\sum_{uv\in E(G)} \sqrt{M_{uv}}
  \le \left(1-\frac1\omega\right)n \cdot \sqrt{\sum_{u, v \in V(G)} M_{uv}}.
\end{equation}
\end{lemma}
Taking $M$ to be the all-ones matrix, \eqref{eq:psd-ms} reads
$2m\le\bigl(1-\frac1\omega\bigr)n^2$, which is Tur\'an's theorem.
Thus \cref{lem:psd} is a weighted strengthening of Tur\'an's theorem,
and we look for a proof of Tur\'an's theorem that is robust under
weights. 

\ppar{The Caro--Wei process}
One of the simplest proofs of Tur\'an's theorem is the random-order
greedy argument underlying the independence bounds of
Caro~\cite{Caro1979} and Wei~\cite{Wei1981}. We first recall the
process, run on the complement so that it grows a clique rather than
an independent set, and then identify the random block partition it
produces.

The process itself has
been studied in the theoretical computer science community: its parallel round
complexity was studied in \cite{BlellochFinemanShun2012} and settled
in \cite{FischerNoever2019}, and the independent set it produces on
sparse graphs was analyzed via local limits
in~\cite{KrivelevichEtAl2024}. The block partition defined below is exactly the clustering that the random-pivot algorithm of Ailon,
Charikar, and Newman~\cite{AilonCharikarNewman2008} for correlation
clustering produces on the complement graph. 

\begin{definition}[Caro--Wei Process]
\label{defn:caro-wei-process}
Given a graph $G$, choose a uniformly random ordering of
$V(G)$.  Starting with $S_0=V(G)$, repeat the following operation.
Let $x_i$ be the first vertex of $S_i$ in the chosen order, which
we call a \emph{pivot}.  Set
\[
  S_{i+1}=S_i\cap N_G(x_i),
  \qquad
  B_i=S_i\setminus S_{i+1},
\]
and stop when $S_i=\varnothing$. The sets $B_0,B_1,\ldots$ are the \emph{blocks} of the process. In every outcome the blocks form a partition
$\mathcal P=\{B_0,B_1,\ldots\}$ of the vertex set $V(G)$. 

The pivots $x_0,x_1,\ldots$ form
a clique, since every later pivot belongs to the neighborhood of
every earlier pivot; in particular there are at most $\omega(G)$
pivots in every outcome. Since the number of blocks is equal to the number of pivots, we conclude that $\mathcal P$ has at most $\omega(G)$ parts.

For $v\in V(G)$ and $uv\in E(G)$, define
\[
  p_G(v)=\Pr(v\text{ is selected as a pivot}),
  \qquad
  q_G(uv)=\Pr(u\text{ and }v\text{ lie in different blocks}).
\]
\end{definition}
In particular, $v$ is selected as a pivot whenever it precedes all of
its non-neighbors in the chosen order, so
$p_G(v)\ge\frac{1}{n-d_G(v)}$.  Taking expectations yields the
Caro--Wei inequality $\omega(G) \geq \mathbb{E}[\#\mathrm{pivots}]\ge\sum_{v}\frac1{n-d_G(v)}$, which
implies Tur\'an's theorem by Cauchy--Schwarz; see
\cite{AlonSpencer2016}.

\ppar{Reduction to a weighted Tur\'an theorem}
Write $W_{\mathcal P}(M)$ for the separated mass with weights
$a_{uv}=M_{uv}$ given by the entries of a matrix $M$.  Since every
outcome of the Caro--Wei partition $\mathcal P$ has $r\le\omega$
parts, for doubly nonnegative $M$ a Cauchy--Schwarz argument (carried
out in \cref{sec:reduction}) gives the deterministic upper bound
\[
  W_{\mathcal P}(M)\le\frac12\Bigl(1-\frac1r\Bigr)\one^{\top}M\one
  \le\frac12\Bigl(1-\frac1\omega\Bigr)\one^{\top}M\one.
\]
Comparing this bound with \eqref{eq:psd-ms} shows that \cref{lem:psd} follows from \cref{thm:partition}, where $\mathcal D_{\mathcal P}$ is taken to
be the distribution of the Caro--Wei partition.

\ppar{A (local) harmonic bound for the Caro--Wei process}
The Caro--Wei process separates each edge $e$ with probability
$q_G(e)$, so the expected separated mass of its partition is
$\mathbb{E}\,W_{\mathcal P}(a)=\sum_ea_e\,q_G(e)$.  The weighted
Cauchy--Schwarz inequality
\[
  \Bigl(\sum_{e\in E(G)}\sqrt{a_e}\Bigr)^{\!2}
  \le\Bigl(\sum_{e\in E(G)}a_e\,q_G(e)\Bigr)
    \Bigl(\sum_{e\in E(G)}\frac1{q_G(e)}\Bigr)
\]
 shows that the single, weight-free \emph{harmonic bound}
\begin{equation}
    \label{eq:harmonic}
  \sum_{e\in E(G)}\frac1{q_G(e)}
  \;\le\;
  \frac{\omega-1}{2\omega}\,n^2
\end{equation}
implies \cref{thm:partition}. Remarkably, the following stronger local harmonic inequality holds.
It yields \eqref{eq:harmonic} upon summing over all vertices, using
$\omega\geq\mathbb{E}[\#\mathrm{pivots}]=\sum_v p_G(v)$ and
Cauchy--Schwarz; see \cref{sec:reduction}.

\begin{lemma}\label{lem:local}
For every graph $G$ on $n\ge1$ vertices and every vertex $v\in V(G)$, we have
\begin{equation}\label{eq:local}
  \frac1{p_G(v)}
  +\sum_{u\in N_G(v)}\frac1{q_G(vu)}
  \le n.
\end{equation}
\end{lemma}

Every term on the left-hand side is finite: $p_G(v)\ge\frac1{n-d_G(v)}$
as noted above, and $q_G(vu)\ge\frac1n$ because $u$ and $v$ are
separated whenever $u$ comes first in the order.
As a sanity check, we explicitly compute this sum for a complete multipartite
graph $G$: in every outcome of the Caro--Wei process, the blocks are exactly the parts of the graph, so every edge is separated. Thus $q_G\equiv1$, and a vertex $v$ in a part of size $s$ is a pivot
precisely when it comes first within its part, so $p_G(v)=1/s$.  Hence equality holds in \eqref{eq:local}:
\[
  \frac1{p_G(v)}+\sum_{u\in N_G(v)}\frac1{q_G(vu)}
  = s+(n-s)=n.
\]
\ppar{Proof idea of the local estimate}
The main strategy is to average the induction hypothesis on each neighborhood. Conditioning on the first vertex of the random order yields the
recurrences (\cref{lem:recurrences}), where $G_x=G[N_G(x)]$:
\[
  n\,p_G(v)=1+\sum_{x\in N_G(v)}p_{G_x}(v),
  \qquad
  n\,q_G(uv)=|N_G(u)\mathbin{\triangle}N_G(v)|
   +\sum_{x\in N_G(u)\cap N_G(v)}q_{G_x}(uv).
\]
Applying the induction
hypothesis to $v$ inside each $G_x$ with $x\in N_G(v)$ and
substituting into the recurrences shows that $n\,p_G(v)$ and the
$n\,q_G(vu)$ dominate the row sums of a positive matrix indexed by
$\{0\}\cup N_G(v)$ with unit diagonal in which every pair of opposite
entries has product at least one; the crucial case of the latter is
the two-variable inequality $(b^{-1}-1+c)(c^{-1}-1+b)\ge1$ for
$b,c\in(0,1]$.  For any such matrix, an AM--GM argument
(\cref{lem:rows}) shows that the reciprocals of the row sums sum to
at most $1$, which is \eqref{eq:local} after multiplying by $n$.

\subsection{Organization}
\cref{sec:local} proves the local estimate: we establish the two
first-pivot recurrences together with an elementary inequality on
reciprocal row sums, and combine them into an inductive proof of
\cref{lem:local} (restated there as \cref{thm:local}).
\cref{sec:reduction} then carries out the chain of reductions
described above: using \cref{lem:local}, we prove
\cref{thm:partition}, then \cref{lem:psd}, and with them
\cref{thm:main}.

\ppar{Formal verification}
The proof of \cref{thm:main} has been formally verified in Lean~4 using Mathlib.  The
complete formalization, together with documentation, reproducible
build instructions, and a kernel-axiom audit, is available at
\href{https://github.com/ShengtongZhang-alt/SqOmega}
{\texttt{github.com/ShengtongZhang-alt/SqOmega}}.

\section{Proof of the local estimate}\label{sec:local}

In this section we prove \cref{lem:local}.  We begin with the two
first-pivot recurrences stated in the overview, then establish an
elementary inequality on reciprocal row sums, and finally combine
them in an induction on the number of vertices.

We begin with the following first-pivot recurrences.

\begin{lemma}\label{lem:recurrences}
If $n=|V(G)|$, then for every $v\in V(G)$,
\begin{equation}\label{eq:pivot-rec}
  n\,p_G(v)
  =1+\sum_{x\in N_G(v)}p_{G[N_G(x)]}(v),
\end{equation}
and for every $uv\in E(G)$,
\begin{equation}\label{eq:sep-rec}
  n\,q_G(uv)
  =|N_G(u)\mathbin{\triangle}N_G(v)|
   +\sum_{x\in N_G(u)\cap N_G(v)}q_{G[N_G(x)]}(uv).
\end{equation}
\end{lemma}

\begin{proof}
  The first pivot $x_0$ is the first vertex of the random order, so it
is uniformly distributed on $V(G)$.  Conditioned on $x_0=x$, the
surviving set is $S_1=N_G(x)$, the order restricted to $N_G(x)$ is
again uniform, and the remaining steps run the Caro--Wei process on
$G[N_G(x)]$.  Hence a vertex of $N_G(x)$ becomes a pivot, or a pair
of vertices of $N_G(x)$ is separated, with the same conditional
probabilities as in the process on $G[N_G(x)]$.

For \eqref{eq:pivot-rec}, we condition on $x_0$.  If $x_0=v$, then
$v$ is a pivot.  If $x_0=x$ for a neighbor $x$ of $v$, then $v$
survives, and it becomes a pivot with conditional probability
$p_{G[N_G(x)]}(v)$.  If $x_0$ is a non-neighbor of $v$, then $v$ is
deleted into the block $B_0$ and is never a pivot.  Averaging the
three cases, each value of $x_0$ carrying weight $\frac1n$, yields
\eqref{eq:pivot-rec}.

For \eqref{eq:sep-rec}, note first that, as neighborhoods are open
and $uv\in E(G)$, the endpoints $u$ and $v$ themselves lie in
$N_G(u)\mathbin{\triangle}N_G(v)$.  If
$x_0\in N_G(u)\mathbin{\triangle}N_G(v)$, then exactly one of the
endpoints survives to $S_1$, while the other is deleted into $B_0$,
so the pair is separated.  If $x_0\in N_G(u)\cap N_G(v)$, then both
endpoints survive, and they are separated with conditional
probability $q_{G[N_G(x_0)]}(uv)$.  In every remaining case both
endpoints are deleted into $B_0$, and the pair is never separated.
Averaging as before yields \eqref{eq:sep-rec}.
\end{proof}

We also need the following elementary inequality on reciprocal row
sums.

\begin{lemma}\label{lem:rows}
Let $R=(r_{ij})_{i,j\in I}$ be a matrix with strictly positive entries,
indexed by a finite nonempty set $I$. Suppose that $r_{ii}=1$ and
$r_{ij}r_{ji}\ge1$ whenever $i\ne j$.  Then
\[
  \sum_{i\in I}\frac{1}{\sum_{j\in I}r_{ij}}\le1.
\]
\end{lemma}

\begin{proof}
Write $c_i=\sum_j r_{ij}$, $y_i=c_i^{-1}$, and $Y=\sum_i y_i$.
Then
\begin{align*}
Y=\sum_i c_i y_i^2=\sum_i y_i^2
    +\sum_{i<j}\bigl(r_{ij}y_i^2+r_{ji}y_j^2\bigr)\ge \sum_i y_i^2+2\sum_{i<j}y_i y_j
   =Y^2,
\end{align*}
where the inequality is AM--GM and $r_{ij}r_{ji}\ge1$.
Thus $Y\le1$.
\end{proof}

We now restate and prove the local harmonic inequality from
\cref{lem:local}.

\begin{theorem}\label{thm:local}
For every graph $G$ on $n\ge1$ vertices and every vertex $v\in V(G)$, we have
\begin{equation}\label{eq:local-restate}
  \frac1{p_G(v)}
  +\sum_{u\in N_G(v)}\frac1{q_G(vu)}
  \le n.
\end{equation}
\end{theorem}

\begin{proof}
We prove \eqref{eq:local-restate} by induction on $n$. The case $n=1$ is immediate since $p_G(v) = 1$ and the second summand is $0$. Now we assume $n \geq 2$ and the theorem is true for all graphs on at most $(n - 1)$ vertices.

Fix $v\in V(G)$ and
put
\[
  U=N_G(v),\qquad d=|U|,\qquad F=G[U].
\]
For each $u\in U$, let
\[
  G_u=G[N_G(u)],
  \qquad
  p_u=p_{G_u}(v),
  \qquad
  t_u=|N_G(u)\setminus(U\cup\{v\})|.
\]
For every oriented edge $ux\in E(F)$, set
\[
  q_{u,x}=q_{G_u}(vx),
\]
the probability that $v$ and $x$ are separated by the Caro--Wei
process on $G_u$.  The direction matters: $q_{u,x}$ is computed in
$G_u$, whereas $q_{x,u}$ is computed in $G_x$.

For $u\in U$, the graph $G_u$ has $1+d_F(u)+t_u<n$ vertices, and the
neighbors of $v$ inside $G_u$ are exactly $N_F(u)$.  Applying the
induction hypothesis to $v$ in $G_u$ gives
\[
  \frac1{p_u}
  +\sum_{x\in N_F(u)}\frac1{q_{u,x}}
  \le 1+d_F(u)+t_u,
\]
or equivalently
\begin{equation}\label{eq:budget}
  t_u\ge
  \frac1{p_u}-1
  +\sum_{x\in N_F(u)}\left(\frac1{q_{u,x}}-1\right).
\end{equation}

By \eqref{eq:pivot-rec},
\begin{equation}\label{eq:root-row}
  n\,p_G(v)=1+\sum_{x\in U}p_x.
\end{equation}
For $u\in U$, the common neighbors of $u$ and $v$ are $N_F(u)$, and
\[
  |N_G(v)\mathbin{\triangle}N_G(u)|
  =|N_G(u)|+|N_G(v)|-2|N_G(u)\cap N_G(v)|= d+1+t_u-d_F(u).
\]
Thus \eqref{eq:sep-rec}, followed by \eqref{eq:budget},
gives
\begin{align}
  n\,q_G(vu)
  &=d+1+t_u-d_F(u)
    +\sum_{x\in N_F(u)}q_{x,u}\notag\\
  &\ge d+\frac1{p_u}-d_F(u)
    +\sum_{x\in N_F(u)}
      \left(\frac1{q_{u,x}}-1+q_{x,u}\right).
  \label{eq:edge-row}
\end{align}

We now package \eqref{eq:root-row}--\eqref{eq:edge-row} as row sums.
Let $I=\{0\}\cup U$ and define a positive matrix $R=(r_{ij})_{i,j\in I}$
by $r_{ii}=1$,
\[
  r_{0x}=p_x,\qquad r_{x0}=p_x^{-1}\qquad(x\in U),
\]
and, for distinct $x,u\in U$,
\[
  r_{xu}=
  \begin{cases}
    1, & xu\notin E(F),\\[2mm]
    q_{x,u}^{-1}-1+q_{u,x}, & xu\in E(F).
  \end{cases}
\]
The opposite entries have product at least one.  This is immediate for
pairs involving $0$ and for nonedges of $F$.  For $xu\in E(F)$,
writing $b=q_{x,u}$ and $c=q_{u,x}$,
\[
  \left(\frac1b-1+c\right)
  \left(\frac1c-1+b\right)-1
  =\frac{(1-b)(1-c)(1+bc)}{bc}\ge0.
\]

Let $s_i=\sum_{j\in I}r_{ij}$.  By \eqref{eq:root-row},
$s_0=n p_G(v)$.  For $u\in U$, there are
$d-1-d_F(u)$ nonneighbors of $u$ in $U\setminus\{u\}$, so
\begin{align*}
  s_u
  &=1+\frac1{p_u}+d-1-d_F(u)
    +\sum_{x\in N_F(u)}
      \left(\frac1{q_{u,x}}-1+q_{x,u}\right)\le n\,q_G(vu).
\end{align*}
The last inequality is \eqref{eq:edge-row}.  Therefore \cref{lem:rows} yields
\[
  \frac1{n p_G(v)}
  +\sum_{u\in U}\frac1{n q_G(vu)}
  \le \frac1{s_0}+\sum_{u\in U}\frac1{s_u}
  \le1.
\]
Multiplying by $n$ completes the proof.
\end{proof}
We remark that the proof of \eqref{eq:local} here is essentially
analytic.  Since \eqref{eq:local} concerns natural statistics
of a single random process, and holds with equality for all complete
multipartite graphs, it would be interesting to find a combinatorial
interpretation.
\section{From the local estimate to a PSD inequality}\label{sec:reduction}
In this section, we prove the chain of implications
\begin{center}
  Lemma~\ref{lem:local} $\Longrightarrow$ Theorem~\ref{thm:partition} $\Longrightarrow$  Lemma~\ref{lem:psd} $\Longrightarrow$  Theorem~\ref{thm:main}
\end{center}
Each implication was sketched in the introduction.
\begin{proof}[Proof of Theorem~\ref{thm:partition}]
Let $\mathcal{P}$ denote the random partition obtained in the Caro--Wei process, and let $R_G$ denote the number of pivots. Since the pivots form a clique, $R_G \leq \omega$ in every outcome; thus, by linearity of expectation,
$$\sum_v p_G(v) = \mathbb{E}R_G \leq \omega.$$
Summing \eqref{eq:local} over all $v \in V$ gives
\[
  \sum_{v \in V} \frac1{p_G(v)}
  +2\sum_{uv\in E(G)}\frac1{q_G(uv)}
  \le n^2.
\]
In addition, the Cauchy--Schwarz inequality shows that
\[
  \sum_{v \in V}\frac1{p_G(v)}
  \ge \frac{n^2}{\sum_v p_G(v)}
  \ge\frac{n^2}{\omega}.
\]
Substitution proves \eqref{eq:harmonic}:
$$\sum_{uv\in E(G)}\frac1{q_G(uv)}
  \le \frac12\left(1-\frac{1}{\omega}\right)n^2.$$
Therefore, the weighted Cauchy--Schwarz inequality gives
\[
  \Bigl(\sum_{e\in E(G)}\sqrt{a_e}\Bigr)^{\!2}
  \le\Bigl(\sum_{e\in E(G)}a_e\,q_G(e)\Bigr)
    \Bigl(\sum_{e\in E(G)}\frac1{q_G(e)}\Bigr)
  \le \frac12\left(1-\frac{1}{\omega}\right)n^2\,
      \mathbb{E}\,W_{\mathcal P}(a),
\]
which rearranges to the claimed bound.
\end{proof}

\begin{proof}[Proof of \cref{lem:psd}]
The edgeless case is trivial, so assume that $G$ is nonempty, and let
$M$ be doubly nonnegative.  Let $\mathcal P$ be the random partition
provided by \cref{thm:partition}, and let $W_{\mathcal P}(M)$ be the
mass of $M$ on the edges of $G$ separated by $\mathcal P$.

We first bound $W_{\mathcal P}(M)$ from above in each outcome.  Fix an
outcome with blocks $B_1,\ldots,B_r$, where $r\le\omega$, and let
$z_i$ be the indicator vector of $B_i$, so that $\sum_{i=1}^r
z_i=\one$.  By entrywise nonnegativity of $M$,
\begin{equation}\label{eq:W-split}
  W_{\mathcal P}(M)
  \le\sum_{i<j}z_i^{\top}Mz_j
  =\frac12\Bigl(\one^{\top}M\one-\sum_{i=1}^r z_i^{\top}Mz_i\Bigr).
\end{equation}
To bound the subtracted sum from below, factor $M=B^{\top}B$ and set
$w_i=Bz_i$.  Applying the Cauchy--Schwarz inequality to the $r$
vectors $w_1,\ldots,w_r$,
\[
  \one^{\top}M\one
  =\Bigl\lVert\sum_{i=1}^r w_i\Bigr\rVert^2
  \le r\sum_{i=1}^r\lVert w_i\rVert^2
  =r\sum_{i=1}^r z_i^{\top}Mz_i.
\]
Substituting $\sum_i z_i^{\top}Mz_i\ge\frac1r\one^{\top}M\one$ into
\eqref{eq:W-split} yields the pointwise bound
\begin{equation}\label{eq:upper-W}
  W_{\mathcal P}(M)
  \le\frac12\Bigl(1-\frac1r\Bigr)\one^{\top}M\one
  \le\frac12\cK\one^{\top}M\one.
\end{equation}
On the other hand, \cref{thm:partition} applied to the edge weights
$a_{uv}=M_{uv}$ gives
\begin{equation}\label{eq:lower-W}
  \mathbb{E}\,W_{\mathcal P}(M)
  \ge\frac{2}{\cK n^2}
     \Bigl(\sum_{uv\in E(G)}\sqrt{M_{uv}}\Bigr)^2.
\end{equation}
Combining \eqref{eq:upper-W} and \eqref{eq:lower-W} and rearranging,
\[
  \Bigl(\sum_{uv\in E(G)}\sqrt{M_{uv}}\Bigr)^2
  \le\frac14\cK^2 n^2\,\one^{\top}M\one,
\]
which is \eqref{eq:psd-ms} after taking square roots.
\end{proof}

\begin{proof}[Proof of \cref{thm:main}]
Write $A=A^+-A^-$, where $A^+,A^-\succeq0$ are the positive and
negative parts of the spectral decomposition of the adjacency matrix
$A$, so that $A^+A^-=0$ and $s^{+}(G)=\normF{A^{+}}^2$.  Put
$M=A^{+}\circ A^{+}$.  Then $M$ is entrywise
nonnegative, and it is positive semidefinite by the Schur product
theorem \cite[Theorem~7.5.3]{HornJohnson2013}; hence $M$ is doubly
nonnegative.  Its total mass is
\[
  \one^{\top}M\one
  =\sum_{u,v\in V}\bigl(A^{+}_{uv}\bigr)^2
  =\normF{A^{+}}^2
  =s^{+}(G).
\]

Next, since $A^+A^-=0$,
\[
  \operatorname{tr}(AA^{+})
  =\operatorname{tr}\bigl((A^+-A^-)A^{+}\bigr)
  =\normF{A^{+}}^2
  =s^{+}(G),
\]
while, since $A$ is a $0/1$ matrix with zero diagonal,
\[
  \operatorname{tr}(AA^{+})
  =\sum_{u\ne v}A_{uv}A^{+}_{uv}
  =2\sum_{uv\in E(G)}A^{+}_{uv}.
\]
Therefore
\[
  2\sum_{uv\in E(G)}\sqrt{M_{uv}}
  =2\sum_{uv\in E(G)}\bigl\lvert A^{+}_{uv}\bigr\rvert
  \ \ge\ \operatorname{tr}(AA^{+})
  =s^{+}(G).
\]
Combining this with \cref{lem:psd} applied to $M$ yields
\[
  s^{+}(G)
  \le 2\sum_{uv\in E(G)}\sqrt{M_{uv}}
  \le \left(1-\frac1\omega\right)n\,\sqrt{\one^{\top}M\one}
  = \left(1-\frac1\omega\right)n\,\sqrt{s^{+}(G)}.
\]
If $s^{+}(G)=0$, the claimed bound is trivial; otherwise dividing
by $\sqrt{s^{+}(G)}$ completes the proof.
\end{proof}

In particular, $n-\sqrt{s^{+}(G)}\ge n/\omega>0$, so \cref{thm:main}
rearranges to $\frac{n}{n-\sqrt{s^{+}(G)}}\le\omega$, proving
\cref{conj:EW}.

\section*{Acknowledgments and AI disclosure}
We used ChatGPT to generate exploratory code for testing conjectures and to help polish the language of proof drafts. All mathematical content was verified by the authors. The core ideas and the overall strategy of the proof were developed independently by the authors; the motivation and the path leading to the argument are described in the overview; any suggestions from the tool were modified or discarded as needed, and the authors take full responsibility for the correctness and originality of the paper.

\end{document}